\documentclass[11pt,letter]{article}

\usepackage{amsmath}
\usepackage{amsthm}
\usepackage{amssymb}
\usepackage{verbatim}
\usepackage{bbm}
\usepackage{algorithm}
\usepackage[margin=1in]{geometry}

\newcommand{\bdisp}{\begin{displaystyle}}
\newcommand{\edisp}{\end{displaystyle}}

\renewcommand{\Pr}{\operatorname*{\mathbb{P}}}

\newcommand{\Exp}{\operatorname*{\mathbb{E}}}
\newcommand{\from}{\leftarrow}

\renewcommand{\vec}[1]{\boldsymbol{\mathbf{#1}}}

\newcommand{\Real}{\mathbb{R}}

\renewcommand{\d}{\mathrm{d}}

\newcommand{\eps}{\epsilon}

\newcommand{\muhat}{\hat{\mu}}
\newcommand{\vhat}{\hat{v}}
\newcommand{\alphahat}{\hat{\alpha}}

\newcounter{nTheorems}

\newtheorem{theorem}[nTheorems]{Theorem}

\newtheorem{lemma}[nTheorems]{Lemma}
\newtheorem{proposition}[nTheorems]{Proposition}

\newtheorem{example}{Example}
\newtheorem{fact}[nTheorems]{Fact}

\theoremstyle{definition}
\newtheorem{definition}[nTheorems]{Definition}

\title{Optimal Sub-Gaussian Mean Estimation in $\mathbb{R}$}
\author{Jasper C.H.~Lee\\ \ \\Brown University\\\texttt{jasperchlee@brown.edu} \and Paul Valiant\\ \ \\IAS \& Purdue University\\\texttt{pvaliant@gmail.com}}
\date{\today}

\begin{document}

\maketitle

\begin{abstract}
    We revisit the problem of estimating the mean of a real-valued distribution, presenting a novel estimator with sub-Gaussian convergence: intuitively, ``our estimator, on \emph{any} distribution, is as accurate as the sample mean is for the Gaussian distribution of matching variance." Crucially, in contrast to prior works, our estimator does not require prior knowledge of the variance, and works across the entire gamut of distributions with finite variance, including those without any higher moments.
    Parameterized by the sample size $n$, the failure probability $\delta$, and the variance $\sigma^2$, our estimator is accurate to within $\sigma\cdot(1+o(1))\sqrt{\frac{2\log\frac{1}{\delta}}{n}}$, tight up to the $1+o(1)$ factor. 
    Our estimator construction and analysis gives a framework generalizable to other problems, tightly analyzing a sum of dependent random variables by viewing the sum implicitly as a 2-parameter $\psi$-estimator, and constructing bounds using mathematical programming and duality techniques.
\end{abstract}

\section{Introduction}

We revisit one of the most fundamental problems in statistics: estimating the mean of a real-valued distribution, using as few independent samples from it as possible.
Our proposed estimator has convergence that is optimal not only in a big-O sense (i.e.~``up to multiplicative constants"), but tight to a $1+o(1)$ factor, under the minimal (and essentially necessary, see below) assumption of the finiteness of the variance.
Previous works, discussed further in Section~\ref{sec:works}, are either only big-O tight~\cite{Jerrum:1986,Nemirovsky:1983,Alon:1999}, or require additional strong assumptions such as the variance being known to the estimator~\cite{Catoni:2012} or assumptions that allow for accurate estimates of the variance, such as the kurtosis (fourth moment) being finite~\cite{Catoni:2012,Devroye:2016}.

\subsection{The Model and Main Result}
\label{sec:model}

Given a set of i.i.d.~samples from a real-valued distribution, the goal is to return, with extremely high probability, an accurate estimate of the distribution's mean. Specifically, given a sample set $X$ of size $n$ consisting of independent draws from a real-value distribution $D$, an $(\epsilon,\delta)$-estimator of the mean is a function $\hat{\mu}:\mathbb{R}^n\rightarrow\mathbb{R}$ such that, except with failure probability $\leq \delta$, the estimate $\hat{\mu}(X)$ is within $\epsilon$ of the true mean $\mu(D)$. Namely, \begin{equation}
\label{eq:PAC}
    \Pr(|\muhat(X) - \mu(D)| \le \eps) \ge 1-\delta 
\end{equation}

The goal is to find the optimal tradeoff between the sample size $n$, and the error parameters $\eps$ and $\delta$, for the distribution $D$. Fixing any two of the three parameters and minimizing the third yields essentially equivalent reformulations of the problem: we can fix $\epsilon,\delta$ and minimize the \emph{sample complexity} $n$; we can fix $\delta,n$ and minimize \emph{error} $\epsilon$; or we can fix $\epsilon,n$ and minimize the \emph{failure probability} $\delta$ (maximizing the \emph{robustness} $1-\delta$).


Perhaps the most standard and well-behaved setting for mean estimation is when the distribution $D$ is a Gaussian.
The sample mean (the empirical mean) is a provably optimal estimator in our sense when $D$ is Gaussian: for any $\epsilon,\delta>0$, the sample mean $\mu(X)$ is an $(\eps,\delta)$-estimator when given a sample set of size $n=(2+o(1))\frac{\sigma^2(D)\cdot\log\frac{1}{\delta}}{\eps^2}$ (all logarithms will be base $e$); and there is \emph{no} $(\eps,\delta)$- estimator for Gaussians if the constant 2 in the previous expression for the sample size is changed to any smaller number.

The main result of this paper is an estimator that performs as well, on \emph{any} distribution with finite variance, as the sample mean does on a Gaussian, without knowledge of the distribution or its variance:

\begin{theorem}\label{thm:main}
Estimator~\ref{alg:merged}, given $\delta,n>0$, defines a function $\muhat$ such that with probability at least $1-\delta$, given a sample set $X$ of size $n$, yields an estimate with error \[|\muhat(X)-\mu(D)| \le \sigma(D)\cdot(1+o(1))\sqrt{\frac{2\log\frac{1}{\delta}}{n}}\]
Here, the $o(1)$ term tends to 0 as  $\left(\frac{\log\frac{1}{\delta}}{n},\delta\right) \to (0,0)$.
Furthermore, as evidenced by the Gaussian case, there is no estimator which, under the same settings, produces an error that improves on our guarantees by more than a $1+o(1)$ multiplicative factor.
\end{theorem}


We have parameterized the above theorem in terms of fixing the sample size $n$ and the robustness parameter $\delta$ and asking for the minimum error $\eps$; however, because of the simple functional form of the bounds of Theorem~\ref{thm:main}, we can equivalently rephrase it as saying that, for any $\eps,\delta$, (a reparameterized) Algorithm~\ref{alg:merged} is an $(\eps,\delta)$ estimator using $(2+o(1))\frac{\sigma(D)^2}{\eps^2}\log\frac{1}{\delta}$ samples; or for any $n,\eps$, Algorithm~\ref{alg:merged} gives an estimate that is $\delta=\exp(-\frac{n\eps^2}{(2+o(1))\cdot\sigma^2(D)})$-robust. For each of these formulations, the performance is optimal up to the $1+o(1)$ factor, as evidenced by the well-known Gaussian case, as explained above.

We make the following observations regarding the main (minimal) assumption in the theorem, namely the finiteness of the variance of the unknown distribution.
First, imposing further assumptions about the finiteness of higher moments will not yield any improvements to the result, since matching lower bounds are provided by Gaussians, for which all moments are finite.
Second, as shown by Devroye et al.~\cite{Devroye:2016}, relaxing the finite variance assumption by only assuming, say, the finiteness of the $(1+\beta)^\text{th}$ moment for some $\beta<1$ will yield strictly worse sample complexity.
In particular, the sample complexity will have an $\eps$-dependence that is $\omega(1/\eps^2)$.
Thus, our result shows that mean estimation can be performed at a sub-Gaussian rate, with the optimal multiplicative constant of $2$ in the sample complexity, if and only if the variance of the underlying distribution is finite.

We also contrast with previous works that attain optimal sub-Gaussian convergence but make additional assumptions such as the finiteness of the kurtosis ($4^\textrm{th}$ moment)~\cite{Catoni:2012,Devroye:2016}.
The gap in assumptions between those works and this work is not only theoretical, but also of practical consequence: power law distributions are known to be good models of certain real-world phenomena, and for exponents in the range $(3,5]$, the variance exists, but not the kurtosis.


\subsection{Our Approach}

We briefly describe the main features of our estimator, as a setting for what follows, and to distinguish it from prior work. At the highest level: in order to return a $\delta$-robust estimate of the mean, our estimator ``throws out the $\frac{1}{3}\log\frac{1}{\delta}$ most extreme points in the samples", and returns the mean of what remains. More specifically, outliers are thrown out in a \emph{weighted} manner, where we throw out a \emph{fraction} of each data point, with the fraction proportional to the square of its distance from a median-of-means initial guess for the mean, where the fraction is capped at 1, and the proportionality constant is chosen so that the total weight thrown out equals exactly $\frac{1}{3}\log\frac{1}{\delta}$.
See Estimator~\ref{alg:merged} for full details, but we stress here that the estimator is simple to implement---it may be computed in linear time---and therefore applicable in practice.

The above description is rather different from the typical M-estimator/$\psi$-estimator approach of Catoni~\cite{Catoni:2012} and other works in this area.
However, as we see in Section~\ref{sec:estimator}, our estimator can be reinterpreted as a 2-parameter $\psi$-estimator, and the proof of our main result will crucially rely on this reformulation.

\subsection{Motivation: $3^\textrm{rd}$-order corrections of the empirical mean}

Perhaps the most non-obvious part of our estimator is throwing out exactly $\frac{1}{3}\log\frac{1}{\delta}$ many samples.
We motivate this quantity in this section, by considering the special case of estimating the mean of asymmetric---very biased---Bernoulli distributions, which is in some sense an extremal case for our setting.

\begin{example}
Consider the mean estimation problem, given $n$ samples from a Bernoulli distribution supported on 0 and 1, where the probability of drawing 1 equals some parameter $p$. Thus the number of 1s observed is distributed as the Binomial distribution $Bin(n,p)$, of mean $np$ and variance $np(1-p)$. The interesting regime for us is when $p$ is very small, and thus $1-p\approx 1$, and the Binomial distribution is essentially the Poisson distribution $Poi(np)$ of mean and variance $\lambda=np$. In this setting, the mean estimation problem becomes: given a sample $k$ from $Poi(np)$, and the parameters $n$ and $\delta$, return an estimate that, except with failure probability $\delta$, is as close as possible to $p$ (or equivalently $np$). Given a Poisson sample $k\from Poi(np)$, returning simply $k$ is a natural estimate of $np$; however, since Poisson distributions are slightly skewed, it turns out that one should instead return the correction $k-\frac{1}{3}\log\frac{1}{\delta}$. 

Explicitly, the Poisson distribution has pmf $poi(\lambda,k)=\frac{\lambda^k e^{-\lambda}}{k!}$, whose logarithm, using Stirling's approximation for the factorial, expanding to $3^\textrm{rd}$ order in $k$, and dropping lower-order terms in $\lambda$ is $-\frac{(k-\lambda)^2}{2\lambda}+\frac{(k-\lambda)^3}{6\lambda}$. The $2^\textrm{nd}$-order term here corresponds to a Gaussian centered at $k=\lambda$ of variance $\lambda$, which is a standard approximation for the Poisson distribution. However, crucially, the $3^\textrm{rd}$ order term, corresponding to the positive skewness of the Poisson distribution, increases the pmf to the right of $k=\lambda$ and decreases it by an essentially symmetric factor to the left. 

Seeking a $\delta$-robust estimation of $\lambda$ from a single sample of $k$, we are concerned, essentially, with the interval where the Poisson pmf is greater than $\delta$, or equivalently, where the log pmf is greater than $\log\delta$. The quadratic $-\frac{(k-\lambda)^2}{2\lambda}$ in the first term of the above approximation equals $\log\delta$ when $k=\lambda\pm \sqrt{2\lambda \log\frac{1}{\delta}}$, and this interval is centered at the $\lambda$. However, crucially, when we take into account the $3^\textrm{rd}$-order term, the interval where $poi(\lambda,k)\geq\delta$ essentially shifts to become $k=\frac{1}{3}\log\frac{1}{\delta}+\lambda\pm \sqrt{2\lambda \log\frac{1}{\delta}}$.
Thus, given a single sample, one can $\delta$-robustly estimate the mean of a Poisson distribution similarly well as the Gaussian of same mean and variance, but only if one returns the sample minus $\frac{1}{3}\log\frac{1}{\delta}$.
\end{example}

Thus, the $\frac{1}{3}\log\frac{1}{\delta}$ term in our estimator arises essentially from a $3^\textrm{rd}$ order correction to the sample mean, at least in the special case of Bernoulli distributions.
For additional intuition and motivation about the ``$3^\textrm{rd}$ order correction" in our estimator, please refer to Appendix~\ref{app:3rd}.

\subsection{Key Contributions in Our Construction and Analysis}\label{sec:contributions}

In addition to settling the fundamental sample complexity question of mean estimation, we point out that the estimator construction and analysis may also be of independent interest. In particular, the analysis framework---as described below---is generalizable to other problem settings and estimator constructions.

Our overall analysis framework may be viewed as a Chernoff bound---showing exponentially small probability of estimation error via bounds on a moment generation function (expectation of an exponentiated real-valued random variable). However, since we seek to analyze our estimator to sub-constant accuracy, many standard approaches fail to yield the required resolution. We point out three crucial components of our approach.

First, our estimator (Estimator~\ref{alg:merged}) is \emph{not} a sum of independent terms, which is fundamental to standard Chernoff bound approaches, and thus we instead reformulate our estimator as a 2-parameter $\psi$-estimator (see Definition~\ref{def:psi}).
This technique rewrites our estimate $\hat{\mu}$ as the first coordinate of the root $(\muhat,\alphahat)$ of a system of 2 equations $\psi_\mu(\muhat, \alphahat) = 0$ and $\psi_\alpha(\muhat, \alphahat) = 0$, where the functions $\psi_\mu(\muhat, \alphahat) = \sum_i \psi_\mu(x_i,\muhat,\alphahat)$ and $\psi_\alpha(\muhat,\alphahat) = \sum_i \psi_\alpha(x_i,\muhat,\alphahat)$ are explicitly sums of a corresponding function applied to each of the $n$ independent data points in the sample set.
Thus we have bought independence at the price of making the estimator an implicit function, introducing two new variables.
One-dimensional estimators of this form are standard: for example, Catoni's~\cite{Catoni:2012} mean estimator in the case of known variance is a (1 parameter) $\psi$-estimator for which he proves finite sample concentration. 
However, adding another dimension---$\alphahat$, a new implicit variable whose value the estimator will ultimately discard---is less standard, without standard analysis techniques, yet significantly increases the expressive power of such estimators~\cite{stefanski}.
Our high-level approach is to find carefully chosen linear combinations of the functions $\psi_\mu$ and $\psi_\alpha$, each of which is now a sum of independent terms, and prove Chernoff bounds about these linear combinations.




Second, even after identifying these linear combinations of $\psi$ functions, the corresponding Chernoff bound analysis is difficult to directly tackle.
The Chernoff bound analysis, as it turns out, is essentially equivalent to bounding a max-min optimization problem where the maximization is over the set of real-valued probability measures with mean 0 and variance 1.
In other words, the max-min optimization problem can be interpreted as having uncountably infinitely many variables.
In order to drastically simplify the problem and make it amenable to analysis, we use convex-concave programming and linear programming duality techniques to reduce the problem to a pure minimization problem with a small finite number of variables, which we can analyze tightly.


We believe that the above two ideas---1) reformulating an estimator as a multi-parameter $\psi$-estimator, so as to find a proxy of the estimator that is a sum of independent variables, and 2) viewing the corresponding Chernoff bound analysis as an optimization problems and applying relevant duality techniques---form a general analysis framework which expands the space of possible estimators that are amenable to \emph{tight} analysis.



\section{Related Work}
\label{sec:works}

There is a long history of work on real-valued mean estimation in a variety of models.
In the problem setting we adopt, where the sole assumption is on the finiteness of the second moment, the median-of-means algorithm~\cite{Jerrum:1986,Nemirovsky:1983,Alon:1999} has long been known to have sample complexity tight to within constant multiplicative factors, albeit with a sub-optimal constant.
Catoni~\cite{Catoni:2012} improved this sample complexity to essentially optimal (tight up to a $1+o(1)$ factor), by focusing on the special cases where the variance of the underlying distribution is known 
or the $4^\text{th}$ moment is finite and bounded (in which case the second moment can be accurately estimated). 
We stress however that the finiteness of the $4^\textrm{th}$ moment is nonetheless a much stronger assumption than our minimal assumption on the finiteness of the variance (see the discussion at the end of Section~\ref{sec:model}).

Moving beyond the original problem formulation, Devroye et al.~\cite{Devroye:2016} drew the distinction between a \emph{single-$\delta$ estimator},  which takes in the robustness parameter $\delta$ as input, versus a \emph{multiple-$\delta$ estimator}, which does not take any $\delta$ as input, but still provides guarantees across a wide range of $\delta$ values.
In their work, making the same finite kurtosis assumption as Catoni, they achieved a multiple-$\delta$ estimator with essentially optimal sample complexity, for a wide range of $\delta$ values.
It is thus natural and prudent to ask whether a multiple-$\delta$ estimator can exist for the entire class of distributions with finite variance, for a meaningful range of $\delta$ values.
Unfortunately, Devroye et al.~\cite{Devroye:2016} showed strong lower bounds answering the question in the negative.
Hence, in this work, our proposed estimator is (and must be) a single-$\delta$ estimator, taking $\delta$ as input.

Many applications have arisen from the success of sub-Gaussian mean estimation, showing how to leverage or extend Catoni-style estimators to new settings, achieving sub-Gaussian performance on problems such as regression, empirical risk minimization, and online learning (bandit settings): for example see~\cite{Minsker:2018,Brownlees:2015,Catoni:2017,Bubeck:2013}.

A separate but closely related line of work is on \emph{high dimensional} mean estimation.
While estimators generalizing the ``median-of-means" construction were found to have statistical convergence tight to multiplicative constants, until recently, such estimators took super-polynomial time to compute~\cite{lugosi2019sub}.
A recent line of work~\cite{Hopkins:2020,Cherapanamjeri:2019,Lei:2020}, started by Hopkins~\cite{Hopkins:2020}, thus focuses on the computational aspect, and brought the computation time first down to polynomial time, with subsequent work bringing it further down to quadratic time using spectral methods.

A recent comprehensive survey by Lugosi and Mendelson~\cite{Lugosi:2019} explains much of the above works in greater detail.

Other works have focused on mean estimation in restrictive settings, for example, with differential privacy constraints.
For example, Kamath et al.~\cite{Kamath:2020} studied the differentially private mean estimation problem in the constant probability regime, and showed strong sample complexity separations from our unrestricted setting.
Duchi, Jordan and Wainright~\cite{Duchi:2013,Duchi:2018} also studied the problem under the stricter constraint of \emph{local} differential privacy.
See the work of Kamath et al.~\cite{Kamath:2020} for a more comprehensive literature review on differentially private mean estimation.


Part of our tight analysis relies on insights from mathematical programming and duality; see~\cite{polyanskiy2019dualizing} for a detailed discussion of prior works that use such mathematical programming and duality tools to either design or analyze statistical estimators~\cite{Polyanskiy:2017,Moitra:2013,Wu:2016,Wu:2019,Jiao:2015,Valiant:2011}.




\section{Our Estimator}
\label{sec:estimator}

In this section, we present our estimator (Estimator~\ref{alg:merged}), 
as well as its reformulation as a 2-parameter $\psi$-estimator. We then present some perspective and basic structural properties of the estimator that will serve as a foundation for the analysis to follow.

\begin{algorithm}
\floatname{algorithm}{Estimator}
\caption{The Main Estimator}
\label{alg:merged}
\vspace*{2mm}
\quad Inputs: \begin{itemize}
    \item $n$ independent samples $\{x_i\}$ from the unknown underlying distribution $D$ (guaranteed to have finite variance)
    \item Confidence parameter $\delta$
                \end{itemize}
\begin{enumerate}
    \item Compute the median-of-means estimate $\kappa$: evenly partition the data into $\log\frac{1}{\delta}$ groups and let $\kappa$ be the median of the set of means of the groups.
    \item Find the solution $\alpha$ to the monotonic, piecewise-linear equation $\sum_i \min(\alpha (x_i - \kappa)^2, 1) = \frac{1}{3}\log\frac{1}{\delta}$
    \item Output: $\muhat = \kappa + \frac{1}{n}\sum_i (x_i-\kappa)(1-\min(\alpha(x_i-\kappa)^2,1))$
\end{enumerate}
\end{algorithm}

\subsection{Meaning of the Estimator}

Consider the expression in Step 3 for the final returned value of the estimator, $\muhat = \kappa + \frac{1}{n}\sum_i (x_i-\kappa)(1-\min(\alpha(x_i-\kappa)^2,1))$. Without the final $\min$ expression, the expression $\kappa + \frac{1}{n}\sum_i (x_i-\kappa)\cdot 1$ computes exactly the sample mean. The factor $(1-\min(\alpha(x_i-\kappa)^2,1))$ may be thought of as a weight on the $i^{\textrm{th}}$ element, between 0 and 1, where a weight of 1 leaves that element as is, but a weight towards 0 essentially throws out part of the sample $x_i$ and instead defaults to the median-of-means estimate $\kappa$. Thus, rather than either keeping or discarding each entry, the weight $\min(\alpha(x_i-\kappa)^2,1)$ specifies what \emph{fraction} of the $i^{\textrm{th}}$ sample to discard.

The condition in Step 2 of Estimator~\ref{alg:merged} picks $\alpha$ so that the total, weighted, number of discarded samples equals $\frac{1}{3}\log\frac{1}{\delta}$. The expression $\min(\alpha(x_i-\kappa)^2,1)$ specifying what fraction of each $x_i$ to discard says, essentially, that this fraction should be proportional to the square of the deviation of $x_i$ from the mean estimate $\kappa$, capped at 1 so that we do not discard ``more than 100\% of" any sample $x_i$.

\subsection{Structural Properties of the Estimator}
\label{sec:basic}

We point out three basic structural properties of Estimator~\ref{alg:merged} that both shed light on the estimator itself, and will be crucial to its analysis.
We formally state and prove these properties in Appendix~\ref{app:basic}.

First, the estimator is ``affine invariant" in the sense that, if its input samples $\{x_i\}$ undergo an affine map $x\rightarrow ax+b$ then its output will be mapped correspondingly.
Second, as is well known, the median-of-means estimate $\kappa$ of Step 1, while not as accurate as what we will eventually return, is robust in the sense that, with probability at least $1-\delta/2$, the median-of-means estimate has additive error from the true mean that is at most $O(\sigma\cdot\sqrt{\frac{\log\frac{1}{\delta}}{n}})$---proportional to the eventual guarantees of our estimator, but with somewhat worse proportionality constant.
Third, if we temporarily ignore Step 1, treating $\kappa$ as a free parameter, we show that the final output of the algorithm, $\muhat$, varies very little with $\kappa$.
Combined with the accuracy guarantees of the median-of-means estimate, the difference in the final estimate between using the median-of-means as $\kappa$ versus using the \emph{true} mean as $\kappa$ is inconsequential (a $o(1)$ factor) compared to the total additive error we aim for.
Therefore, for the purposes of \emph{analysis}, it suffices to assume that $\kappa$ takes the value of the true mean (though an algorithm could not do this in practice, as the true mean is unknown).

These structural properties allow us to drastically simplify the analysis: the affine invariance means it is sufficient to show our estimator works for the special case of distributions with mean 0 and variance 1; the second and third properties mean that errors in $\kappa$ effectively do not matter, and, for distributions with mean 0, it is sufficient to omit Step 1 and instead just analyze the case where $\kappa=0$.

We point out that Estimator~\ref{alg:merged} when modified to set $\kappa=0$ (independently of the samples) is \emph{no longer} affine invariant, nor is its reformulation as a $\psi$-estimator in Section~\ref{sec:psi-formulation}.
The structural properties in this section show that, instead of analyzing the actual estimator (Estimator~\ref{alg:merged} which is affine invariant), it suffices to analyze this artificially simplified, although no longer affine invariant, estimator which sets $\kappa = 0$, on distributions with mean 0 and variance 1. Explicitly, in the rest of the paper we will show Proposition~\ref{prop:main} (Section~\ref{sec:analysis}), which analyzes the mean-0 variance-1 case of the $\psi$-estimator defined below in Definition~\ref{def:psi}; the discussion of this section---made formal in Appendix~\ref{app:basic}---shows that this proposition implies our main result, Theorem~\ref{thm:main}.

\subsection{Representing a Special Case of Estimator~\ref{alg:merged} as a $\psi$-Estimator}
\label{sec:psi-formulation}
As discussed in Section~\ref{sec:contributions}, our estimator, even its simplified version with $\kappa = 0$, is not a sum of independent terms, making it difficult to tightly bound its moment generating function, and hence also difficult to prove its concentration around the true mean using a Chernoff-style bound.
Our solution is to reformulate Estimator~\ref{alg:merged}, with the simplifying assumption that $\kappa = 0$, as a 2-parameter $\psi$-estimator, as defined in Definition~\ref{def:psi}.
This reformulation defines our estimate $\muhat$ implicitly in terms of two new functions $\psi_\mu$ and $\psi_\alpha$ that are indeed sums of $n$ independent terms, each term depending on a single $x_i$.
We will use this representation crucially for the concentration analysis of the estimator.

\begin{definition}
\label{def:psi}
Consider Estimator~\ref{alg:merged} but with Step 1 replaced with ``$\kappa = 0$".
The estimator can be equivalently expressed as follows:
\begin{enumerate}
    \item Input: $n$ independent samples $X = x_1,\ldots,x_n$
    \item Solve for the (unique) pair $(\muhat,\alphahat)$ satisfying $\psi_\mu = 0$ and $\psi_\alpha = 0$, where the functions are defined as follows:
    \begin{align*}
        \psi_\mu(X,\muhat,\alphahat) &= \sum_{i=1}^n \left(\muhat - x_i\left(1-\min\left(\alphahat x_i^2,1\right)\right)\right)\\
        \psi_\alpha(X,\muhat,\alphahat) &= \sum_{i=1}^n \left(\min\left(\alphahat x_i^2,1\right) - \frac{1}{3n}\log\frac{1}{\delta} \right)
    \end{align*}
    (Note that $\alphahat > 0$ always)
    \item Output: $\muhat$ from the previous step
\end{enumerate}
We will sometimes omit $\muhat$ from the arguments of $\psi_\alpha$ since $\muhat$ is not used in the definition of the function. We will often refer to the pair $(\psi_\mu,\psi_\alpha)$ as a 2-element vector $\vec{\psi}$.

For convenience in the rest of the paper, we define $\vhat\equiv\frac{\log(1/\delta)}{3n\alphahat}$, which we refer to as the ``truncated empirical variance"; this is because, if we modify the $\psi_\alpha=0$ condition by removing the ``truncation" of taking the min with 1, then the resulting condition, when expressed in terms of $\vhat=\frac{\log(1/\delta)}{3n\alphahat}$ and rearranged, is exactly the condition that $\vhat$ is the empirical variance: $\frac{1}{n}\sum_{i=1}^n x_i^2$. Thus $\alphahat$ may be thought of as a proxy for the empirical variance, as $\vhat=\frac{\log(1/\delta)}{3n\alphahat}$ equals the empirical variance, except in cases when samples are far enough from 0 that they are ``truncated" by the ``$\min$".
\end{definition}

Interestingly, in the case that none of the samples are ``truncated", (and $\kappa=0$), the overall output of the estimator becomes $\frac{1}{n}\sum_i x_i-\alpha x_i^3=\frac{1}{n}\sum_i x_i-  \frac{\log(1/\delta)}{3n\vhat}x_i^3$, namely, $\muhat$ is ``the empirical mean, corrected by subtracting $\frac{1}{3n}\log\frac{1}{\delta}$ times the ratio of the empirical $3^\textrm{rd}$ moment over the empirical $2^\textrm{nd}$ moment."

\begin{proof}[Proof that Definition~\ref{def:psi} is equivalent to Estimator~\ref{alg:merged} when $\kappa$ is set to 0]
Fix a set of samples $X = \{x_i\}$.
We observe that Estimator~\ref{alg:merged}, with the additional simplifying assumption that $\kappa = 0$, can be represented by the following 2 equations.
\begin{equation}
\label{eq:simpleest}
\begin{aligned}
    &\sum_i \min(\alpha \, x_i^2, 1) = \frac{1}{3}\log\frac{1}{\delta}\\
    &\muhat = \frac{1}{n}\sum_i x_i(1-\min(\alpha\,x_i^2,1))
\end{aligned}
\end{equation}
Estimator~\ref{alg:merged} solves for $\alpha$ in the first line, and uses this $\alpha$ value to compute the estimate $\muhat$ in the second line. The two conditions of Equation~\ref{eq:simpleest} are equivalent to the two conditions $\psi_\alpha=0$, $\psi_\mu=0$ respectively, and thus the two estimators are equivalent.
\end{proof}

\section{Analyzing our estimator}
\label{sec:analysis}

In this section, we outline the proof of our main theorem, restated as follows.

\vspace*{3mm}\noindent\textbf{Theorem~\ref{thm:main}.}
\emph{Estimator~\ref{alg:merged}, given $\delta,n>0$,  and a sample set $X$ of $n$ independent samples from distribution $D$, will, with probability at least $1-\delta$ over the sampling process, yield an estimate $\muhat$ with error at most $|\muhat(X)-\mu(D)| \le \sigma(D)\cdot(1+o(1))\sqrt{\frac{2\log\frac{1}{\delta}}{n}}$.
Here, the $o(1)$ term tends to 0 as  $\left(\frac{\log\frac{1}{\delta}}{n},\delta\right) \to (0,0)$.
}
\medskip

The discussion of the structural properties of Estimator~\ref{alg:merged} in Section~\ref{sec:basic} shows that it is sufficient to instead show that, for any distribution of mean 0 and variance 1, the $\psi$-estimator of Definition~\ref{def:psi} will return an estimate $\muhat$ that is close to 0, except with tiny probability.
(See Appendix~\ref{app:basic} for the formal statements of the claims of Section~\ref{sec:basic}.) Recall also that, since the $\psi$-estimator solves for $(\muhat,\alphahat)$ such that  $\vec{\psi}(X,\muhat,\alphahat) = 0$ (where $X$ is the sample set) and returns $\muhat$, the claim that the returned estimate will be close to 0 is equivalent to saying that \emph{every} $(\muhat,\alphahat)$ pair with $\muhat$ \emph{far} from 0 must violate the equation, namely  $\vec{\psi}(X,\muhat,\alphahat) \neq 0$.
We thus prove the following proposition (Proposition~\ref{prop:main}), to yield Theorem~\ref{thm:main}.
Note that the failure probability in Proposition~\ref{prop:main} is $\delta/2$ (instead of $\delta$, as in Theorem~\ref{thm:main}), accounting for an additional $\delta/2$ probability that the median-of-means estimate in Step 1 of Estimator~\ref{alg:merged} fails. 

\begin{proposition}
\label{prop:main}
There exists a universal constant $c > 0$ such that, fixing $\eps' = \left(1+\frac{c\log\log\frac{1}{\delta}}{\log\frac{1}{\delta}}\right)\sqrt{\frac{2\log\frac{1}{\delta}}{n}}$, we have that for all distributions $D$ with mean 0 and variance 1, with probability at least $1-\frac{\delta}{2}$ over the set of samples $X$, for all $\muhat, \alphahat$ where $|\muhat| > \eps'$ and $\alphahat>0$, the vector $\vec{\psi}(X,\muhat,\alphahat)\neq 0$.
\end{proposition}

Proposition~\ref{prop:main} asks us to show that, with high probability,  $\vec{\psi}(X,\muhat,\alphahat)$ is not at the origin for \emph{any} choice of $|\muhat|>\eps',\alphahat$; instead, as a proof strategy, we choose a finite bounded mesh of $\muhat,\alphahat$ and show that the function $\vec{\psi}(X,\muhat,\alphahat)$ is 1) not just nonzero, but far from the origin on this set, 2) Lipschitz in between mesh elements, and 3) monotonic (in an appropriate sense) outside the mesh bounds.
Step 1), discussed below, contains the most noteworthy part of the proof, a mathematical programming-inspired bound to help complete a delicate Chernoff bound argument.

For simplicity, we reparameterize to work with $\vhat\equiv\frac{\log(1/\delta)}{3n\alphahat}$ (the ``truncated empirical variance") instead of $\alphahat$: the mesh we analyze, covering the most delicate region for analysis, will span the interval $\vhat\in[0.05,55.5]$, namely, where the truncated empirical variance $\vhat$ is within a constant factor of the true variance of 1. Note that this should \emph{not} be taken to imply that $\vhat\in[0.05,55.5]$ with high probability---the truncated empirical variance is not designed to be a good estimate of the variance, merely as a step in robustly estimating the mean; and further, accurate estimates of the variance are simply impossible in general without further assumptions such as bounds on the distribution's $3^\textrm{rd}$ or $4^\textrm{th}$ moments.
We also want to distinguish our estimator from Catoni's~\cite{Catoni:2012}: Catoni's estimator relies on having a high-precision estimate of the variance (to within a $1+o(1)$ factor) in order to achieve the desired performance.
By contrast, our estimator is robust against wild inaccuracies of the (truncated) empirical variance $\vhat$ compared to the true variance of 1. In short, the approach of our estimator should be viewed as distinct from Catoni's, since, while Catoni's estimator relies on an initial good guess at the variance, ours thrives in the inevitable situations where $\vhat$ is far from 1.

We return to describing our strategy for analyzing the performance of our estimator.
For each $\muhat,\vhat=\frac{\log(1/\delta)}{3n\alphahat}$ that we analyze (from the finite mesh): instead of directly showing that, with $\geq 1-\frac{\delta}{2}$ probability, $\vec{\psi}(X,\muhat,\alphahat)$ is far from the origin in some direction, we instead \emph{linearize} this claim; we prove the stronger claim that there exists a specific direction $\vec{d}(\vhat)$ such that with $\geq 1-\frac{\delta}{2}$ probability, $\vec{\psi}(X,\muhat,\alphahat)$ is more than  $\frac{1}{\log(1/\delta)}$ distance from the origin in direction $\vec{d}$ (specifically we lower bound the dot product $\vec{d}(\vhat)\cdot\vec{\psi}(\muhat,\alphahat)$, while we upper bound each coordinate of $\vec{d}$ inversely with the Lipschitz coefficients of $\vec{\psi}$).
The crucial advantage of this reformulation is that, since each of $\psi_\mu,\psi_\alpha$ is a sum of $n$ terms, that are each a function of an independent sample $x_i$ from $D$, the  dot product $\vec{d}(\vhat)\cdot\vec{\psi}(X,\muhat,\alphahat)$ is thus also a sum of $n$ independent terms, and thus we finish the proof with a Chernoff bound, Lemma~\ref{lem:Chernoff}.
The Chernoff bound argument itself is standard; however, to bound the resulting expression requires an extremely delicate analysis that we pull out into a separate 4-variable inequality expressed as Lemma~\ref{lem:quadratic}---see the discussion around the lemma for more details and for motivation of the analysis from a mathematical programming perspective.



We state the crucial Chernoff bound (Lemma~\ref{lem:Chernoff}) and the Lipschitz bounds (Lemma~\ref{lem:vhat}), and then use them to prove Proposition~\ref{prop:main}.
We prove Lemmas~\ref{lem:Chernoff} and~\ref{lem:vhat} in the next section, along with the statement and proof of the delicate component that is Lemma~\ref{lem:quadratic}.

\begin{lemma}
\label{lem:Chernoff}
Consider an arbitrary distribution $D$ with mean 0 and variance 1.
There exists a universal constant $c$ where the following claim is true.
Fixing $\muhat = \eps'=\left(1+\frac{c\log\log \frac{1}{\delta}}{\log\frac{1}{\delta}}\right)\sqrt{\frac{2\log\frac{1}{\delta}}{n}}$, then for all $\delta$ smaller than some universal constant, and for all $\vhat \in [0.05, 55.5]$, there exists a vector $\vec{d}(\vhat)$ where $d_\mu\geq 0$, and both $\sqrt{\frac{n}{\log(1/\delta)}}|d_\mu|,|d_\alpha|$ are bounded by a universal constant, such that 
\begin{equation*}
\Pr_{X \from D^n}\left(\vec{d}(\vhat)\cdot \vec{\psi}\left(X, \muhat = \eps', \alphahat=\frac{\log(1/\delta)}{3n\vhat}\right) > \frac{1}{\log\frac{1}{\delta}}\right) \ge 1-\frac{\delta}{\log^4\frac{1}{\delta}}\end{equation*}
Furthermore, for $\vhat=0.05$ we have $d_\mu=\sqrt{3.75\frac{\log(1/\delta)}{n}}$, $d_\alpha=\sqrt{3}$; and for $\vhat=55.5$ we have $d_\mu=0$, $d_\alpha<0$.
\end{lemma}

\begin{lemma}
\label{lem:vhat}
Consider an arbitrary set of $n$ samples $X$. Consider the expressions $\psi_\mu(X,\muhat,\alphahat),\psi_\alpha(X,\alphahat)$, reparameterized in terms of $\vhat\equiv\frac{\log(1/\delta)}{3n\alphahat}$ in place of $\alphahat$.
Suppose the equation $\psi_\alpha(X,\alphahat) = 0$ has a solution in the range $\vhat \in [0.05,55.5]$. Then the functions $\sqrt{\frac{\log(1/\delta)}{n}}\psi_\mu(X,\muhat,\alphahat)$ and $\psi_\alpha(X,\alphahat)$ are Lipschitz with respect to $\vhat$ on the entire interval $\vhat \in [0.05,55.5]$, with Lipschitz constant $c\log\frac{1}{\delta}$ for some universal constant $c$.
\end{lemma}

We now prove Proposition~\ref{prop:main}, which per our previous discussion, implies our main result, Theorem~\ref{thm:main}.

\begin{proof}[Proof of Proposition~\ref{prop:main}]
As in Lemma~\ref{lem:Chernoff}, we fix $\eps'=(1+\frac{c\log\log \frac{1}{\delta}}{\log\frac{1}{\delta}})\sqrt{\frac{2\log\frac{1}{\delta}}{n}}$, where $c$ is some universal constant.

By symmetry, instead of considering positive and negative $\muhat$, it suffices to consider the case $\muhat>\eps'$ (as opposed to $\muhat < -\eps'$) and show that this case succeeds with probability at least $1-\frac{\delta}{4}$.

To prove the claim, we first prove a stronger statement on a restricted domain, that with probability at least $1-\frac{\delta}{4}$ over the randomness of the sample set $X$, for each $\vhat\in[0.05,55.5]$ there exists a vector $\vec{d}=(d_\mu,d_\alpha)$ such that $\vec{d}\cdot\vec{\psi}(X,\eps',\alphahat\equiv\frac{\log(1/\delta)}{3n\vhat})>0$, with $d_\mu\geq 0$ throughout, and, for $\vhat=0.05$ we have $d_\mu=\sqrt{3.75\frac{\log(1/\delta)}{n}}$, $d_\alpha=\sqrt{3}$; and for $\vhat=55.5$ we have $d_\mu=0$, $d_\alpha<0$.

We will first apply Lemma~\ref{lem:Chernoff} to each $\vhat$ in a discrete mesh: let $M$ consist of evenly spaced points between $0.05$ and $55.5$ with spacing $1/\log^3 \frac{1}{\delta}$ (thus with $\Theta(\log^3\frac{1}{\delta})$ many points).

By Lemma~\ref{lem:Chernoff} and a union bound over these $\Theta(\log^3\frac{1}{\delta})$ points, we have that with probability at least $1-\frac{\delta}{\Theta(\log\frac{1}{\delta})}$ (which is at least $1-\frac{\delta}{4}$ for $\delta$ smaller than some universal constant) over the set of $n$ samples $X$, for all $\vhat \in M$, there exists a vector $\vec{d}(\vhat)$ such that $\vec{d}(\vhat) \cdot \vec{\psi}(X,\muhat = \eps',\alphahat\equiv\frac{\log(1/\delta)}{3n\vhat}) > 1/\log\frac{1}{\delta}$, where $\vec{d}$ further satisfies the desired positivity and boundary conditions, and where both $\sqrt{\frac{n}{\log(1/\delta)}}|d_\mu|,|d_\alpha|$ are bounded by a universal constant. For the rest of the proof, we will only consider sets of samples $X$ satisfying the above condition. 

Now consider an arbitrary $\vhat' \in [0.05,55.5]\setminus M$ and consider the vector $\vec{\psi}$ evaluated at $\alphahat'=\frac{\log(1/\delta)}{3n\vhat'}$.
We wish to extend the dot product inequality to hold also for $\vhat'$.
If $\psi_\alpha\neq 0$ then there is nothing to prove: set $d_\mu=0$ and $d_\alpha=\mathrm{sign}(\psi_\alpha)$; otherwise, $\psi_\alpha=0$ means we may apply Lemma~\ref{lem:vhat} to conclude that both $\sqrt{\frac{\log(1/\delta)}{n}}\psi_\mu(X,\muhat,\alphahat')$ and $\psi_\alpha(X,\muhat,\alphahat')$ are Lipschitz with respect to $\vhat'$ on the interval $\vhat'\in[0.05,55.5]$, with Lipschitz constant $c\log\frac{1}{\delta}$ for some universal constant $c$.

Consider the closest $\vhat \in M$ to $\vhat'$, which by definition of $M$ is at most $1/\log^3\frac{1}{\delta}$ away.
By assumption on $X$, there exists a vector $\vec{d}$ such that $\vec{d}\cdot\vec{\psi}(X,\muhat=\eps',\alphahat=\frac{\log(1/\delta)}{3n\vhat}) > 1/\log\frac{1}{\delta}$, with $d_\mu\geq 0$ and both $\sqrt{\frac{n}{\log(1/\delta)}}|d_\mu|,|d_\alpha|$ are bounded by a universal constant. Because of the Lipschitz bounds on $\vec{\psi}$, combined with the bounds on the size of the $d_\mu,d_\alpha$, we conclude that the Lipschitz constant of the dot product (treating the vector $\vec{d}$ as fixed) is $O(\log\frac{1}{\delta})$.
Thus, the large positive dot product at $\vhat$ implies at least a positive dot product nearby at $\vhat'$: $\vec{d}\cdot\vec{\psi}(X,\muhat=\eps',\vhat') > \frac{1}{\log\frac{1}{\delta}} - O(\log\frac{1}{\delta})\frac{1}{\log^3\frac{1}{\delta}} > 0$, for sufficiently small $\delta$ as given in the proposition statement.

Having shown the stronger version of the claim for the restriction $\muhat=\eps'$ and $\vhat\in[0.05,55.5]$ we now extend to the entire domain via three monotonicity arguments.
Explicitly, assume the set of samples $X$ satisfies the dot product inequality above with the vector function $\vec{d}(\vhat)$, where $\vec{d}(\vhat)$ satisfies the boundary conditions at $\vhat = 0.05$ and $55.5$ specified in Lemma~\ref{lem:Chernoff}.
From this assumption, we will show that $\vec{\psi} \neq 0$ for \emph{any} positive $\vhat=\frac{\log(1/\delta)}{3n\alphahat}$, and for \emph{any} $\muhat\geq\eps'$.


First consider $\vhat>55.5$ (still fixing $\muhat = \eps'$). The function $\psi_\alpha=\sum_{i=1}^n \left(\min\left(\alphahat x_i^2,1\right) - \frac{1}{3n}\log\frac{1}{\delta} \right)$ is an increasing function of $\alphahat$, and thus a decreasing function of $\vhat\equiv\frac{\log(1/\delta)}{3n\alphahat}$.
Since for $\vhat=55.5$, the dot product $\vec{d}\cdot\vec{\psi}>0$ with $d_\mu=0,d_\alpha<0$, the dot product will thus remain positive for this same choice of $\vec{d}$ as we increase $\vhat$ from $55.5$.

Next, for $\vhat<0.05$ (again still fixing $\muhat = \eps'$), we analogously show that the dot product of $\vec{\psi}(X,\eps',\alphahat\equiv\frac{\log(1/\delta)}{3n\vhat})$ with the fixed vector $\vec{d}(0.05)$ will increase as we decrease $\vhat$.
The $i^\textrm{th}$ term in the sums defining $\psi_\mu$ or $\psi_\alpha$ depends on $\alphahat$ (and thus $\vhat$) only in the factor $\min(\alphahat x_i^2,1)$.
Further, there is no dependence unless the first term attains the min, namely $|x_i|\leq\sqrt{1/\alphahat}$, which in turn is upper bounded by $\sqrt{0.15\frac{n}{\log(1/\delta)}}$ because of our assumption that $\vhat < 0.05$.
Thus, the only $i^\text{th}$ terms in the dot product which have $\alphahat$ dependent are simply equal to $d_\mu \alphahat x_i^3 + d_\alpha \alphahat x_i^2=\alphahat x_i^2 (d_\alpha+x_i d_\mu)$.
By our choice of $d_\mu(0.05)=\sqrt{3.75\frac{\log(1/\delta)}{n}} $ and $d_\alpha(0.05)=\sqrt{3}$ from Lemma~\ref{lem:Chernoff}, the expression $(d_\alpha+x_i d_\mu)\geq \sqrt{3}-\sqrt{0.15}\sqrt{3.75}$ is thus always non-negative, and thus the overall dot product cannot decrease as we send $\alphahat$ to $\infty$---equivalently, sending $\vhat$ to 0---as desired. 

We have thus shown that, for all non-negative $\alphahat=\frac{\log(1/\delta)}{3n\vhat}$, there is a vector $\vec{d}$ with $d_\mu\geq 0$ whose dot product with $\vec{\psi}(X,\eps',\alphahat)$ is greater than 0. 
We complete the proof by noting that the only dependence on $\muhat$ in $\vec{\psi}$ is that $\psi_\mu$ is (trivially) increasing in $\muhat$.
Since $d_\mu \geq 0$, increasing $\muhat$ from $\eps'$ will only increase the dot product, and thus the dot product remains strictly greater than 0, implying that $\vec{\psi}(X,\muhat,\alphahat) \ne 0$ as desired.
\end{proof}

\section{Proofs of Lemmas~\ref{lem:Chernoff} and~\ref{lem:vhat}}

The main purpose of this section is to present and motivate the proof of Lemma~\ref{lem:Chernoff}---since our results are tight across such a wide parameter space, the resulting inequalities are somewhat subtle. After, we also present the short proof of Lemma~\ref{lem:vhat}.

\vspace*{3mm}\noindent\textbf{Lemma~\ref{lem:Chernoff}.}
\emph{
Consider an arbitrary distribution $D$ with mean 0 and variance 1.
There exists a universal constant $c$ where the following claim is true.
Fixing $\muhat = \eps'=\left(1+\frac{c\log\log \frac{1}{\delta}}{\log\frac{1}{\delta}}\right)\sqrt{\frac{2\log\frac{1}{\delta}}{n}}$, then for all $\delta$ smaller than some universal constant, and for all $\vhat \in [0.05, 55.5]$, there exists a vector $\vec{d}(\vhat)$ where $d_\mu\geq 0$, and both $\sqrt{\frac{n}{\log(1/\delta)}}|d_\mu|,|d_\alpha|$ are bounded by a universal constant, such that 
}
\begin{equation}\label{eq:chernoff-lemma}\Pr_{X \from D^n}\left(\vec{d}(\vhat)\cdot \vec{\psi}\left(X, \muhat = \eps', \alphahat=\frac{\log(1/\delta)}{3n\vhat}\right) > \frac{1}{\log\frac{1}{\delta}}\right) \ge 1-\frac{\delta}{\log^4\frac{1}{\delta}}\end{equation}
\emph{Furthermore, for $\vhat=0.05$ we have $d_\mu=\sqrt{3.75\frac{\log(1/\delta)}{n}}$, $d_\alpha=\sqrt{3}$; and for $\vhat=55.5$ we have $d_\mu=0$, $d_\alpha<0$.}\medskip

We start the analysis via standard Chernoff bounds  on the complement of the probability in Equation~\ref{eq:chernoff-lemma} via Lemma~\ref{lem:chernoff-start}, before pausing to discuss how mathematical programming and duality insights lead to the formulation of the crucial Lemma~\ref{lem:quadratic}; we then complete the proof.

\begin{lemma}\label{lem:chernoff-start}
Consider an arbitrary distribution $D$ with mean 0 and variance 1. For all sufficiently small $\delta$, for any $\muhat,\alphahat$ and vector $\vec{d}=(d_\mu,d_\alpha)$, we have \[\Pr_{X \from D^n}\left(\vec{d}\cdot \vec{\psi}\left(X, \muhat, \alphahat\right) \leq \frac{1}{\log\frac{1}{\delta}}\right)\leq 2 \left(e^{-d_\mu\muhat+d_\alpha\frac{1}{3n}\log\frac{1}{\delta}}\Exp_{x \from D} (e^{d_\mu x(1-\min(\alphahat x^2,1))-d_\alpha\min(\alphahat x^2,1)})\right)^n\]
\end{lemma}

\begin{proof} 
We upper-bound the probability by exponentiating the negation of both sides  of the expression inside the probability, and then using Markov's inequality:
\allowdisplaybreaks
\begin{align}&\;\hspace{4mm}\Pr_{X \from D^n}\left(\vec{d}(\vhat)\cdot \vec{\psi}(X, \muhat, \alphahat) \leq \frac{1}{\log\frac{1}{\delta}}\right)\nonumber\\
&=\Pr_{X \from D^n}\left(e^{-\vec{d}(\vhat)\cdot \vec{\psi}(X, \muhat, \alphahat)} \geq e^{-\frac{1}{\log\frac{1}{\delta}}}\right)\nonumber\\
&\leq 2\Exp_{X \from D^n}\left(e^{-\vec{d}(\vhat)\cdot \vec{\psi}(X, \muhat, \alphahat)}\right)\quad\text{by Markov's inequality; and  $e^{\frac{1}{\log(1/\delta)}}\leq 2$ for sufficiently small $\delta$}\nonumber\\
&=2 \Exp_{x \from D} (e^{-\vec{d}(\vhat)\cdot\vec{\psi}(x,\muhat,\alphahat)})^n\quad\text{by independence}\nonumber\\
&=2 \left(e^{-d_\mu\muhat+d_\alpha\frac{1}{3n}\log\frac{1}{\delta}}\Exp_{x \from D} (e^{d_\mu x(1-\min(\alphahat x^2,1))-d_\alpha\min(\alphahat x^2,1)})\right)^n\quad\text{substituting the definition of $\vec{\psi}$}\label{eq:chernoff}
\end{align}
\end{proof}

\subsection{Mathematical Programming and Duality Analysis}\label{sec:duality}
In order to show Lemma~\ref{lem:Chernoff}, we aim to find bounds on the failure probability that are as strong as possible. Appealing to Lemma~\ref{lem:chernoff-start} that we have just proven, recall that, as in the standard Chernoff bound methodology, we are still free to choose the parameters $d_\mu,d_\alpha$, which we do so as to minimize the resulting bound on the failure probability. Phrased abstractly, the goal is, for the $\muhat,\alphahat$ of Lemma~\ref{lem:Chernoff}, to show that, for any distribution $D$ of mean 0 and variance 1, there is a choice $\vec{d}=(d_\mu,d_\alpha)$ that makes Equation~\ref{eq:chernoff} sufficiently small. Phrased as an optimization problem, our goal is to evaluate (or tightly bound):

\begin{equation}\label{eq:max-min}\max_D\min_{\vec{d}=(d_\mu,d_\alpha)}
e^{-d_\mu\muhat+d_\alpha\frac{1}{3n}\log\frac{1}{\delta}}\Exp_{x \from D} (e^{d_\mu x(1-\min(\alphahat x^2,1))-d_\alpha\min(\alphahat x^2,1)})\end{equation}
where $D$ ranges over distributions of mean 0 and variance 1.

We will use convex-concave programming and linear programming duality to significantly simplify the max-min program in Equation~\ref{eq:max-min} before we dive into the part of analysis that is ad hoc for this problem.
We wish to emphasize here again that the steps of 1) writing an estimator as a multi-parameter $\psi$-estimator and finding an analogous lemma to our Lemma~\ref{lem:Chernoff}, then 2) using mathematical programming duality to simplify the Chernoff bound analysis, are a framework generalizable for tightly analyzing other estimators.

For simplicity of exposition, assume that we restrict the support of $D$ to some sufficiently fine-grained finite set, meaning that the maximization in Equation~\ref{eq:max-min} is now finite-dimensional, albeit an arbitrarily large finite number.
For each support element $x$, let $D_x$ be a variable representing the probability of choosing $x$ under distribution $D$. The expectation component of Equation~\ref{eq:max-min} may now be expressed as sum that is a linear function in the variables $D_x$:

\begin{equation}\label{eq:max-min-finite}\max_D\min_{\vec{d}=(d_\mu,d_\alpha)}
e^{-d_\mu\muhat+d_\alpha\frac{1}{3n}\log\frac{1}{\delta}}\sum_x D_x \cdot e^{d_\mu x(1-\min(\alphahat x^2,1))-d_\alpha\min(\alphahat x^2,1)}\end{equation}

Using the standard max-min inequality (a form of weak duality in optimization), we have that Equation~\ref{eq:max-min-finite} is upper bounded by swapping the maximization and minimization (Equation~\ref{eq:min-max}), meaning that the vector $\vec{d}$ no longer depends on the distribution $D$.
\begin{equation}
    \label{eq:min-max}
    \min_{\vec{d}=(d_\mu,d_\alpha)}\max_D
e^{-d_\mu\muhat+d_\alpha\frac{1}{3n}\log\frac{1}{\delta}}\sum_x D_x \cdot e^{d_\mu x(1-\min(\alphahat x^2,1))-d_\alpha\min(\alphahat x^2,1)}
\end{equation}
Crucially, however, Equation~\ref{eq:min-max} is not just an upper bound on Equation~\ref{eq:max-min-finite}, but is in fact \emph{equal} to it, due to Sion's minimax theorem~\cite{Sion:1958}.
To apply Sion's minimax theorem, it suffices to check that 1) both $\vec{d}$ and $D$ are constrained to be in convex sets, at least one of which is compact, 2) the objective is convex in $\vec{d}$ and 3) concave in the variables $D_x$.
For the first condition, we note that the set of distributions on a finite domain is compact. The objective is convex in $\vec{d}$ since the objective is the sum of exponentials that are each linear in $\vec{d}$. And the objective is concave in $D_x$ because it is in fact a linear function of $D$.

The guarantee of Sion's minimax theorem means that we may work with Equation~\ref{eq:min-max} instead of Equation~\ref{eq:max-min-finite} without sacrificing tightness in our analysis. This justifies why we are free to choose $\vec{d}=(d_\mu,d_\alpha)$ in Lemma~\ref{lem:Chernoff} that does not depend on the distribution $D$.

To further simplify the problem in Equation~\ref{eq:min-max}, we note again that both the objective and the constraints on $D$ are linear in the variables $D_x$, meaning that the inner maximization is in fact a linear program.
We can then apply linear programming (strong) duality to yield the following equivalent optimization (Equation~\ref{eq:min-min}).
We note that, as above, for the purposes of upper bounding Equation~\ref{eq:max-min}, it suffices to only use weak duality.
Strong duality however guarantees that this step does not introduce slack into the analysis.

The three variables $V,M,S$ in the inner minimization below are the dual variables corresponding to the three constraints on distribution $D$ originally: that $D$ has variance 1, mean 0, and total probability mass 1.

\vspace*{-2mm}
\begin{equation}
\label{eq:min-min}
    \begin{aligned}
        & \min_{\vec{d}=(d_\mu,d_\alpha)} \min_{V, M, S} V+S\\
        \text{for all $x$:}\quad & Vx^2 + Mx + S \ge e^{-d_\mu\muhat+d_\alpha\frac{1}{3n}\log\frac{1}{\delta}+d_\mu x(1-\min(\alphahat x^2,1))-d_\alpha\min(\alphahat x^2,1)}
    \end{aligned}
\end{equation}

We have thus reduced the infinite-dimensional optimization problem of Equation~\ref{eq:max-min} to the five-dimensional problem of Equation~\ref{eq:min-min} (or six dimensions, if we include the universal quantification for $x\in\Real$), a significant simplification.
We bound Equation~\ref{eq:min-min} by explicitly choosing values for $\vec{d}=(d_\mu,d_\alpha),V, M, S$ as functions of $\alphahat,n,\log\frac{1}{\delta}$, and showing that they jointly satisfy the constraint of Equation~\ref{eq:min-min}, for all $x$. We factor out the terms in the exponential that do not depend on $x$; we make the variable substitutions $y\equiv\sqrt{\alphahat}x$ and $\vhat\equiv\frac{\log(1/\delta)}{3n\alphahat}$ to replace dependence on $\alphahat,n,\log\frac{1}{\delta}$ with dependence on the single variable $\vhat$; taking the log of both sides (and swapping sides) yields an expression that is recognizable in the following lemma, where the multipliers of $1,y,y^2$ respectively on the right hand side are essentially our choices of $S,M,V$:

\begin{lemma}
\label{lem:quadratic}
For all $\vhat \in [0.05,55.5]$, there exist $a>0$ and $b$ such that
$$ \forall y\in\mathbb{R}:\,a y\left(1-\min\left(y^2,1\right)\right) - b\cdot\min\left(y^2,1\right) \le \log\left(1+a y+y^2\vhat(-3+\frac{a\sqrt{6}}{\sqrt{\vhat}}-b)\right)$$
where $a\in[C,C']$ and $b\in[-C',C']$ for positive constants $C,C'$.
Further, for $\vhat=0.05$, the pair $a=0.75,b=\sqrt{3}$ works.
\end{lemma}

We emphasize that the application of Lemma~\ref{lem:quadratic} in the proof of Lemma~\ref{lem:Chernoff} below is straightforward, though finding the particular form of Lemma~\ref{lem:quadratic} is not.
Further, one would not seek a result of the form of Lemma~\ref{lem:quadratic} without the guarantees of this section, derived via duality and mathematical programming, showing that ``results of the form of Lemma~\ref{lem:quadratic} encompass the full power of the Chernoff bounds of Equation~\ref{eq:chernoff}." See the end of Section~\ref{sec:Chernoff-proof} for the proof of Lemma~\ref{lem:quadratic}.

\subsection{Proof of Lemma~\ref{lem:Chernoff}}
\label{sec:Chernoff-proof}

We now prove Lemma~\ref{lem:Chernoff} by combining the Chernoff bound analysis of Lemma~\ref{lem:chernoff-start} with the inequality from Lemma~\ref{lem:quadratic}.
We point out that the proof below is direct, without any reference to duality or mathematical programming; however, the discussion of Section~\ref{sec:duality} was crucial to discovering the right formulation for Lemma~\ref{lem:quadratic}.
We prove Lemma~\ref{lem:quadratic} at the end of the section.

\vspace*{3mm}\noindent\textbf{Lemma~\ref{lem:Chernoff}.}
\emph{
Consider an arbitrary distribution $D$ with mean 0 and variance 1.
There exists a universal constant $c$ where the following claim is true.
Fixing $\muhat = \eps'=\left(1+\frac{c\log\log \frac{1}{\delta}}{\log\frac{1}{\delta}}\right)\sqrt{\frac{2\log\frac{1}{\delta}}{n}}$, then for all $\delta$ smaller than some universal constant, and for all $\vhat \in [0.05, 55.5]$, there exists a vector $\vec{d}(\vhat)$ where $d_\mu\geq 0$, and both $\sqrt{\frac{n}{\log(1/\delta)}}|d_\mu|,|d_\alpha|$ are bounded by a universal constant, such that 
}
\begin{equation*}
\Pr_{X \from D^n}\left(\vec{d}(\vhat)\cdot \vec{\psi}\left(X, \muhat = \eps', \alphahat=\frac{\log(1/\delta)}{3n\vhat}\right) > \frac{1}{\log\frac{1}{\delta}}\right) \ge 1-\frac{\delta}{\log^4\frac{1}{\delta}}\end{equation*}
\emph{Furthermore, for $\vhat=0.05$ we have $d_\mu=\sqrt{3.75\frac{\log(1/\delta)}{n}}$, $d_\alpha=\sqrt{3}$; and for $\vhat=55.5$ we have $d_\mu=0$, $d_\alpha<0$.}

\begin{proof}
Start with the bound on the probability of failure given by Lemma~\ref{lem:chernoff-start}: \[2 \left(e^{-d_\mu\muhat+d_\alpha\frac{1}{3n}\log\frac{1}{\delta}}\Exp_{x \from D} (e^{d_\mu x(1-\min(\alphahat x^2,1))-d_\alpha\min(\alphahat x^2,1)})\right)^n\]

For $\vhat\in[0.05,55.5)$ we bound the exponential inside the expectation via the exponential of Lemma~\ref{lem:quadratic}; we also use Lemma~\ref{lem:quadratic} to choose $d_\mu,d_\alpha$ for us (the $\vhat=55.5$ case is covered at the end). Namely, in Lemma~\ref{lem:quadratic} use $\vhat$ as given, substitute $x\sqrt{\alphahat}\equiv y$ (where $\alphahat\equiv\frac{\log(1/\delta)}{3n\vhat}$ as always), and choose $d_\mu\equiv a\sqrt{\alphahat}$, and $d_\alpha \equiv b$---in particular, for $\vhat=0.05$ this gives $d_\mu(0.05)=0.75\sqrt{\alphahat}=0.75\sqrt{\frac{\log(1/\delta)}{3n\vhat}}=\sqrt{3.75\frac{\log(1/\delta)}{n}}$. Thus the failure probability  is bounded by

\allowdisplaybreaks
\begin{align*}\label{eq:chernoff3}&\;\hspace{4mm}2 \left(e^{-d_\mu\muhat+d_\alpha\frac{1}{3n}\log\frac{1}{\delta}}\Exp_{\substack{x \from D\\y=x\sqrt{\alphahat}}} \left(1+a y+y^2\vhat\left(-3+\frac{a\sqrt{6}}{\sqrt{\vhat}}-b\right)\right)\right)^n\\
&=2 \left(e^{-d_\mu\muhat+d_\alpha\frac{1}{3n}\log\frac{1}{\delta}}\left(1+ \frac{\log\frac{1}{\delta}}{3n}\left(-3 + 3d_\mu\sqrt{\frac{2n}{\log(1/\delta)}} - d_\alpha\right)\right)\right)^n\quad\text{since $D$ has mean 0, variance 1}\\
&\leq 2 \left(e^{-d_\mu\muhat+d_\alpha\frac{1}{3n}\log\frac{1}{\delta}+ \frac{\log\frac{1}{\delta}}{3n}\left(-3 + 3d_\mu\sqrt{\frac{2n}{\log(1/\delta)}} - d_\alpha\right)}\right)^n\quad\text{since $1+z\leq e^z$ for any $z$}\\&\leq 2 e^{-d_\mu\sqrt{\frac{2n}{\log(1/\delta)}}c\log\log\frac{1}{\delta}-\log\frac{1}{\delta}}\quad\text{substituting $\muhat = \eps'=\left(1+\frac{c\log\log \frac{1}{\delta}}{\log\frac{1}{\delta}}\right)\sqrt{\frac{2\log\frac{1}{\delta}}{n}}$}\\
&\leq \frac{\delta}{\log^4\frac{1}{\delta}}\quad\text{as desired, for large enough c, since $d_\mu\sqrt{\frac{n}{\log(1/\delta)}}=\frac{a}{\sqrt{3\vhat}}$ is greater than some positive constant.}
\end{align*}

We prove the $\vhat=55.5$ case now. We choose $\d_\mu=0$ and $d_\alpha=-4$, substituting into the bound of Equation~\ref{eq:chernoff} to yield

\begin{align*}2 \left(e^{-\frac{4}{3n}\log\frac{1}{\delta}}\Exp_{x \from D} (e^{4\min(\alphahat x^2,1)})\right)^n&\leq 2\delta^{4/3}\Exp_{x\from D} (1+54\alphahat x^2)^n\quad\text{since for $y\in[0,1]$, $e^{4y}\leq 1+54 y$}\\
&=2\delta^{4/3} (1+55.5\alphahat)^n\quad\text{since $D$ has variance 1}\\
&\leq 2\delta^{4/3} e^{n\cdot 54\frac{\log(1/\delta)}{3\cdot55.5 n}}\quad\text{since $1+z\leq e^z$, substituting def. of $\alphahat$}\\
&=2\delta^{4/3}\delta^{-\frac{54}{3\cdot 55.5}}\leq 2\delta^{1.009}
\end{align*}
which is bounded as desired for small enough $\delta$.
\end{proof}

We now prove Lemma~\ref{lem:quadratic}.

\vspace*{3mm}\noindent\textbf{Lemma~\ref{lem:quadratic}.}
\emph{For all $\vhat \in [0.05,55.5]$, there exist $a>0$ and $b$ such that}
\begin{equation}\label{eq:quadratic-inequality}\forall y\in\mathbb{R}:\,a y\left(1-\min\left(y^2,1\right)\right) - b\cdot\min\left(y^2,1\right) \le \log\left(1+a y+y^2\vhat(-3+\frac{a\sqrt{6}}{\sqrt{\vhat}}-b)\right)\end{equation}
\emph{where $a\in[C,C']$ and $b\in[-C',C']$ for positive constants $C,C'$.
Further, for $\vhat=0.05$, the pair $a=0.75,b=\sqrt{3}$ works.
}

\begin{proof}
We first prove the special case of 1) $\vhat = 0.05$, before moving to the general case of 2) $\vhat \in (0.05, 55.5]$.
We note that our choice of $a(\vhat),b(\vhat)$ is \emph{not} continuous in $\vhat$ at $0.05$, but the usage of the lemma does not require any continuity.
We choose $a,b$ at the edge case $\vhat=0.05$ for convenience.

{\bf 1)} For $\vhat = 0.05$, we choose $a=0.75,b=\sqrt{3}$. This special case of Equation~\ref{eq:quadratic-inequality} simplifies to: \[\forall y\in\mathbb{R}:\,0.75 y\left(1-\min\left(y^2,1\right)\right) - \sqrt{3}\cdot\min\left(y^2,1\right) \le \log\left(1+0.75 y+0.174y^2\right)\]
(where 0.174 is a lower bound on $\hat{v}(-3+\frac{a\sqrt{6}}{\sqrt{\hat{v}}}-b)$ ). This is a 1-dimensional bound and can be easily analyzed in many ways. For the range $y\in[-1,1]:$ the right hand side is at least $\log(1+0.75y)$, which in this range is at least $.75y-.75 y^2$, which is easily shown to be greater than the polynomial expression that the left hand side reduces to in this range, $0.75y-\sqrt{3}y^2-0.75 y^3$. For the remaining range, $y\notin[-1,1]$, the left hand side is the constant $-\sqrt{3}$, and it is easy to check that the quadratic in the argument of the right hand side, $1+0.75y+0.174y^2$, always exceeds $e^{-\sqrt{3}}$.

{\bf 2)} To show Equation~\ref{eq:quadratic-inequality} for the rest of the range of $\vhat \in (0.05, 55.5]$, we choose $a$ to be the positive root of the quadratic equation $\sqrt{\vhat}(a^2-12)+\sqrt{6}a=0$ and let $b=3-a^2/2$---we will see the motivation for this choice shortly. For now, note that the definition of $a$ implies $a\leq\sqrt{12}$, for otherwise $\sqrt{\vhat}(a^2-12)+\sqrt{6}a$ would be greater than 0.

Our proof will analyze the sign of the derivative with respect to $y$ of the difference between the right and left hand sides of Equation~\ref{eq:quadratic-inequality}. For the critical region $|y|\leq 1$ this derivative equals:

\begin{equation}\label{eq:deriv}\frac{a+2y\hat{v}(-3+\frac{a\sqrt{6}}{\sqrt{\hat{v}}}-b)}{1+a y+y^2\hat{v}(-3+\frac{a\sqrt{6}}{\sqrt{\hat{v}}}-b)}-a+3a y^2 +2b y\end{equation}

The crucial step is to choose $a$ to be the positive root of the quadratic equation $\sqrt{\vhat}(a^2-12)+\sqrt{6}a=0$ and let $b=3-a^2/2$, after which Equation~\ref{eq:deriv} miraculously factors as

\[\frac{1}{3a^2}\cdot\frac{y(y+\frac{2}{a})(y+\frac{2}{a}-\frac{a}{3})^2}{y^2+(\frac{4}{a}-\frac{a}{3})y+(\frac{4}{a^2}-\frac{1}{3})}
\]

From this expression for the derivative, it is straightforward to read off its sign. The discriminant of the quadratic in the denominator is $\frac{1}{9}(a^2-12)>0$, meaning the denominator is always positive. The squared term in the numerator cannot affect the overall sign. Thus the sign of the derivative equals the sign of $y(y+\frac{2}{a})$, meaning that the difference between the right and left side of Equation~\ref{eq:quadratic-inequality} is monotonically increasing for $y>0$, and unimodal for $y<0$, having non-positive derivative for $y\in [\frac{2}{a},0]$ and nonnegative derivative for smaller $y$. Thus to show the inequality holds for all $y\in[-1,1]$ it suffices to check it at $y=0$ and $y=-1$.

The $y=0$ case is trivial as both sides of Equation~\ref{eq:quadratic-inequality} equal 0.

For $y=-1$, Equation~\ref{eq:quadratic-inequality}, after expressing both $\sqrt{\hat{v}}$ and $b$ in terms of $a$ becomes


\begin{equation}\label{eq:neg-one-case}\frac{a^2}{2}-3 \le \log\left(-2+a-\frac{36}{a^2-12}\right)\end{equation}

For $a\in[0,\sqrt{12})$, the inverse of the rational expression inside the log is bounded by its linear approximation, $\frac{\sqrt{12}-a}{\sqrt{12}}$. Calling this a new variable $z=\frac{\sqrt{12}-a}{\sqrt{12}}$, which is between 0 and 1, Equation~\ref{eq:neg-one-case} becomes the claim that $6(1-z)^2-3\leq -\log z$, which is easily verified for $z\in(0,1]$.

Lastly, we show Equation~\ref{eq:quadratic-inequality} for $|y| > 1$.
Reexpressing $b$ and $\sqrt{\hat{v}}$ in terms of $a$, the left hand side of the inequality is the constant value  $-b=-3+\frac{a^2}{2}$ (independent of $y$), while the right hand side is $\log(1+a y+\frac{3a^2}{12-a^2}y^2)$.
Analyzing the quadratic inside the log shows that the right hand side has a minimum of $\frac{a^2}{12}$, attained at $y = -\frac{12-a^2}{6a}$.

When the location of this minimum, $y = -\frac{12-a^2}{6a}$, is inside the interval $[-1,1]$, then because this quadratic is monotonic to either side of the minimum, the fact that we have already proven Equation~\ref{eq:quadratic-inequality} for $y=\pm 1$ implies the inequality holds for all $y$ further from 0.

The remaining case is when the minimum is not in $[-1,1]$, namely, $-\frac{12-a^2}{6a}<-1$, meaning $a<1.59$; since $a$ is monotonic in $\hat{v}$, $a$ is at least its value when $\hat{v}=0.05$, namely $a\geq 1.003$. Equation~\ref{eq:quadratic-inequality} thus reduces to showing that, for $a\in[1.003,1.59]$ we have $\frac{a^2}{2}-3\leq\log\frac{a^2}{12}$, which is trivially implied, substituting $z=\frac{a^2}{12}$, by the inequality $6z-3\leq\log z$ for $z\in [0.083,0.22]$, yielding the claim.
\end{proof}

\subsection{Proof of Lemma~\ref{lem:vhat}}
\vspace*{3mm}\noindent\textbf{Lemma~\ref{lem:vhat}.}
\emph{
Consider an arbitrary set of $n$ samples $X$. Consider the expressions $\psi_\mu(X,\muhat,\alphahat),\psi_\alpha(X,\alphahat)$, reparameterized in terms of $\vhat\equiv\frac{\log(1/\delta)}{3n\alphahat}$ in place of $\alphahat$.
Suppose the equation $\psi_\alpha(X,\alphahat) = 0$ has a solution in the range $\vhat \in [0.05,55.5]$. Then the functions $\sqrt{\frac{\log(1/\delta)}{n}}\psi_\mu(X,\muhat,\alphahat)$ and $\psi_\alpha(X,\alphahat)$ are Lipschitz with respect to $\vhat$ on the entire interval $\vhat \in [0.05,55.5]$, with Lipschitz constant $c\log\frac{1}{\delta}$ for some universal constant $c$.}

\begin{proof}
Consider the $\vhat$ derivative of $\psi_\alpha(X,\muhat,\alphahat\equiv\frac{\log(1/\delta)}{3n\vhat}) = \sum_{i=1}^n \left(\min\left(\frac{\log(1/\delta)}{3n\vhat} x_i^2,1\right) - \frac{1}{3n}\log\frac{1}{\delta} \right)$.
The $\vhat$ derivative of $\min\left(\frac{\log(1/\delta)}{3n\vhat} x_i^2,1\right)$ is either $-\frac{\log(1/\delta)}{3n\vhat^2} x_i^2=-\frac{1}{\vhat}\alphahat x_i^2$ or 0, depending on which term in the min is the smallest, and in either case has magnitude at most $\frac{1}{\vhat}\min(\alphahat x_i^2,1)$. Thus the overall $\vhat$ derivative of $\psi_\alpha(X,\muhat,\alphahat)$ has magnitude at most $\frac{1}{\vhat}\sum_i \min(\alphahat x_i^2,1)$. Since, we are guaranteed that $\sum_{i=1}^n \min\left(\alphahat x_i^2,1\right) = \frac{1}{3}\log\frac{1}{\delta}$ for some $\vhat\in [0.05, 55.5]$, we thus have that the derivative is within a constant factor of this across the entire range, as desired.

Similarly, consider the $\vhat$ derivative of $\psi_\mu(X,\muhat,\alphahat) = \sum_{i=1}^n \left(\muhat - x_i\left(1-\min\left(\alphahat x_i^2,1\right)\right)\right)$. The $i^\textrm{th}$ term of this is the $\vhat$ derivative of $\min(\alphahat x_i^3,x_i)$, which is either $-\frac{1}{\vhat}\alphahat x_i^3$ or 0 depending on whether $x_i\leq\sqrt{1/\alphahat}$, and thus the magnitude of this derivative may be bounded by $\frac{1}{\vhat\sqrt{\alphahat}}\sum_{i=1}^n \min\left(\alphahat x_i^2,1\right)$. Since $\sum_{i=1}^n \min\left(\alphahat x_i^2,1\right)$ is bounded by a constant times $\log\frac{1}{\delta}$ (as in the last paragraph), and $\frac{1}{\vhat\sqrt{\alphahat}}$ is bounded by a constant times $\frac{1}{\sqrt{\vhat\alphahat}}=\sqrt{\frac{3n}{\log(1/\delta)}}$, the magnitude of the derivative of $\sqrt{\frac{\log(1/\delta)}{n}}\psi_\mu(X,\muhat,\alphahat)$ is bounded by a constant times $\log\frac{1}{\delta}$, as desired.
\end{proof}

\newpage
\bibliography{mean}
\bibliographystyle{plain}

\newpage
\appendix

\section{Additional ``$3^\textrm{rd}$ Order" Motivation for Our Estimator}
\label{app:3rd}

In this appendix, we give additional motivation of our estimator as a ``$3^\textrm{rd}$ order correction" to the sample mean.

Suppose (for this section only), as in~\cite{Catoni:2012}, that one knows the variance $\sigma^2(D)$ of the distribution in question, or has a good estimate of it. 


\begin{example}\label{ex:empirical-correction}
Given samples $x_1,\ldots,x_n$ from a distribution of mean 0 and variance 1 and bounded higher moments, suppose our goal is to construct a slight variant of the empirical mean that will robustly return an estimate that is close to 0, the true mean; we consider estimates of the form $\frac{1}{n}\sum_{i=1}^n (x_i + c(x_i))$ for some function $c:\mathbb{R}\rightarrow\mathbb{R}$. Explicitly, given a bound $b$, we want our estimate to be between $\pm b$, with as high probability as possible. For simplicity we will consider the positive case, namely, bounding $\Pr_{x_1,\ldots,x_n}(\frac{1}{n}\sum_i (x_i+c(x_i)) \geq b)$. With a view towards deriving a Chernoff bound, we rearrange, multiply by an arbitrary positive constant $\alpha$, and exponente inside the probability to yield that this probability equals $\Pr_{x_1,\ldots,x_n}(\exp(\alpha\sum_i (x_i+c(x_i)-b)) \geq 1)$; by Markov's inequality, this probability is at most $\Exp_{x_1,\ldots,x_n}(\exp(\alpha\sum_i (x_i+c(x_i)-b)))$, for our choice of $\alpha>0$. We set $\alpha=b$. Since each $x_i$ is independent, this probability becomes $\Exp_{x}(\exp(b (x+c(x)-b)))^n$.

Considering the empirical estimator, where $c(x_i)=0$, we thus have that the probability the empirical mean estimate exceeds $b$ is at most the $n^{\textrm{th}}$ power of $\Exp_{x}(\exp(b x-b^2))$, where this expression can be expanded to $3^\textrm{rd}$ order as \[e^{-b^2}\left(1+b \Exp(x) +\frac{1}{2}b^2\Exp(x^2)+\frac{1}{6}b^3\Exp(x^3)+O(x^4)\right)\]
As we assumed the data distribution has mean 0 and variance 1, we can simplify the above expression to \[e^{-b^2}\left(1 +\frac{1}{2}b^2+\frac{1}{6}b^3\Exp(x^3)+O(b^4)\right)\]

Ignoring, for the moment, the $3^\textrm{rd}$ or higher-order terms, this expression is $e^{-b^2}(1 +\frac{1}{2}b^2)\approx e^{-b^2/2}$, whose $n^{\textrm{th}}$ power equals $e^{-b^2n/2}$, which is exactly the bound one would expect for the standard \emph{Gaussian} case, of the probability that the empirical mean of $n$ samples is more than $b$ from the true value. However, the $3^\textrm{rd}$ order term is a crucial obstacle here, as the third moment $\Exp(x^3)$ could be of either sign, skewing either the left tail or right tail to have substantially more mass than in our benchmark of the Gaussian case.

We thus choose a correction function $c(x_i)$ so as to cancel out this $3^\textrm{rd}$-order term and improve the estimate in this regime: to cancel out the term $\frac{1}{6}b^3\Exp(x^3)$ in the $3^\textrm{rd}$-order expansion of our Chernoff bound $\Exp_{x}(\exp(b (x-b)))$, we replace $x$ by $x-\frac{1}{6}x^3b^2$, yielding a bound on the failure probability of the $n^\textrm{th}$ power of \[e^{-b^2}\left(1 +\frac{1}{2}b^2+O(b^4)\right)=e^{-b^2/2+O(b^4)}\] as desired.

For the sake of clarity, we can change variables, letting the leading term of our probability bound $e^{-b^2n/2}$ equal $\delta$, and thus the correction $-\frac{1}{6}x^3 b^2$ becomes $c(x)\equiv-\frac{1}{n}x^3\frac{1}{3}\log\frac{1}{\delta}$, meaning the correction amounts essentially to a $3^\textrm{rd}$ moment correction, split $n$ ways and scaled by the same $\frac{1}{3}\log\frac{1}{\delta}$ of our main algorithm.

We explicitly relate this estimator to Estimator~\ref{alg:merged} by pointing out that, when none of the samples $x_i$ are ``truncated" by Estimator~\ref{alg:merged} (namely, $\alphahat x_i^2\leq 1$ always),  then $\alphahat\equiv\frac{\log(1/\delta)}{3n\vhat}$ may be expressed in terms of the empirical variance $\vhat$; taking $\kappa=0$ for simplicity, the returned estimate will be $\frac{1}{n}\sum_i x_i-\alpha x_i^3 = \frac{1}{n}\sum_i x_i-\frac{1}{\vhat n} x_i^3 \frac{1}{3}\log\frac{1}{\delta}$, which equals the above-derived ``$3^\textrm{rd}$-order corrected estimator" when the empirical variance is the true variance, 1.
\end{example}

In the above example we showed that Chernoff bounds for the empirical mean deteriorate for distributions with large $3^\textrm{rd}$ moments (skew), and that adding a $3^\textrm{rd}$-order correction to the empirical mean corrects for this, leaving essentially ``Gaussian-like" performance. These calculations motivate several features of Estimator~\ref{alg:merged}---including the $\frac{1}{3}\log\frac{1}{\delta}$ parameter, and the $3^\textrm{rd}$-order terms in the expression for $\muhat$---even though the overall form of Estimator~\ref{alg:merged} is rather different, as it must work in all regimes and not just in the cartoon asymptotic regime considered in this example.

\section{Proposition~\ref{prop:main} implies Theorem~\ref{thm:main}}\label{app:basic}
For completeness' sake, we explicitly state and prove the properties described in Section~\ref{sec:basic}, which combine to show that Proposition~\ref{prop:main} implies Theorem~\ref{thm:main}.

\begin{lemma}
\label{lem:unit}
Suppose $X$ is a set of samples in $\Real$.
Then for any $\delta > 0$ and any scale $a>0$ and shift $b$,
$$ \muhat(aX+b,\delta) = a\,\muhat(X,\delta) + b $$
where $\muhat$ denotes the output of Estimator~\ref{alg:merged}.
\end{lemma}

The above lemma follows trivially from the fact that the median-of-means estimate also respects shift and scale in the input samples, and that $\alpha$ is chosen in Step 2 of Estimator~\ref{alg:merged} so that $\min(\alpha (x_i - \kappa)^2, 1)$ does not depend on the affine parameters $a,b$.

\begin{fact}[\cite{Hsu:2016}]
\label{fact:MoM}
For any distribution $D$ with mean $\mu$ and standard deviation $\sigma$, the median-of-means estimate $\kappa$, on input $n$ samples, satisfies
$$ \Pr\left(|\kappa-\mu| > O\left(\sigma\sqrt{\frac{\log\frac{1}{\delta}}{n}}\right)\right) \le \delta $$
\end{fact}

\begin{lemma}
\label{lem:kappa}
Consider a fixed sample set $X$ of size $n$, and a confidence parameter $\delta$.
Let $e(X,\delta,\kappa)$ denote Estimator~\ref{alg:merged} but where Step 1 is omitted and $\kappa$ is instead considered as an input.
Then,
$$ \left|\frac{\d\,e(X,\delta,\kappa)}{\d\,\kappa}\right| = O\left(\sqrt{\frac{\log\frac{1}{\delta}}{n}}\right) $$
\end{lemma}

Fact~\ref{fact:MoM} shows that, except with $\delta$ probability, the median-of-means estimate is within $O\left(\sigma\sqrt{\frac{\log\frac{1}{\delta}}{n}}\right)$ of the true mean, and multiplying this by the Lipschitz constant $O\left(\sqrt{\frac{\log\frac{1}{\delta}}{n}}\right)$ from Lemma~\ref{lem:kappa} shows that the change in output of Algorithm~\ref{alg:merged}, between using the median-of-means versus setting $\kappa=0$, has magnitude $O\left(\sigma\sqrt{\frac{\log\frac{1}{\delta}}{n}}^2\right) = o\left(\sigma\sqrt{\frac{\log\frac{1}{\delta}}{n}}\right)$.
This discrepancy is therefore a $o(1)$ fraction of the additive error guaranteed by Theorem~\ref{thm:main}.

We now prove Lemma~\ref{lem:kappa}.

\begin{proof}
We compute the derivatives with respect to $\alpha$ and $\kappa$ of the $\muhat$ (computed in Step 3 of Estimator~\ref{alg:merged}), and the expression on the left hand side of Step 2, which we denote $\nu\equiv\sum_i \min(\alpha (x_i-\kappa)^2,1)$. We note that for terms where $\min(\alpha(x_i-\kappa)^2,1)=1$, all derivatives are 0, so we adopt the notation ``$\Sigma_<$" to denote summing only over those indices $i$ for which $\alpha (x_i-\kappa)^2<1$. Thus we have 
\begin{align*}
\frac{d\nu}{d\kappa}&=2\sum_< \alpha(x_i-\kappa)\\
\frac{d\nu}{d\alpha}&=\sum_< (x_i-\kappa)^2\\
\frac{d\muhat}{d\kappa}&=1+\frac{1}{n}\sum_< (-1+3\alpha(x_i-\kappa)^2)\\
\frac{d\muhat}{d\alpha}&=-\frac{1}{n}\sum_< (x_i-\kappa)^3
\end{align*}

Recall that $\alpha$ is defined implicitly so as to make the expression $\nu=\frac{1}{3}\log\frac{1}{\delta}$; thus in Estimator~\ref{alg:merged}, if we change $\kappa$ at a rate of 1, then $\alpha$ also changes at rate $\left.-\frac{d\nu}{d\kappa}\right/\frac{d\nu}{d\alpha}$ to keep $\nu$ unchanged. Thus, the overall derivative of the estimate with respect to changing $\kappa$ equals $\frac{d\muhat}{d\kappa}-\frac{d\muhat}{d\alpha}\left.\frac{d\nu}{d\kappa}\right/\frac{d\nu}{d\alpha}$. We bound this from the derivatives computed above.

To bound $\frac{d\muhat}{d\kappa}$, we note that the number of indices \emph{not} in the sum ``$\sum_<$" is at most $\frac{1}{3}\log\frac{1}{\delta}$ because each such $i$ contributes 1 to the left hand side of the condition in Step 2 of Estimator~\ref{alg:merged} and the right hand side equals $\frac{1}{3}\log\frac{1}{\delta}$. Thus the initial terms of $\frac{d\muhat}{d\kappa}$ are bounded as $1+\frac{1}{n}\sum_<(-1)\leq \frac{1}{3n}\log\frac{1}{\delta}$. The remaining part of $\frac{d\muhat}{d\kappa}$, namely $\frac{1}{n}\sum_< 3\alpha(x_i-\kappa)^2$ is $\frac{3}{n}$ times the corresponding terms in $\nu\leq\frac{1}{3}\log\frac{1}{\delta}$ itself, and thus is at most $\frac{1}{n}\log\frac{1}{\delta}$. Thus $\frac{d\muhat}{d\kappa}=O(\frac{1}{n}\log\frac{1}{\delta})$.

We now bound the remaining term $-\frac{d\muhat}{d\alpha}\left.\frac{d\nu}{d\kappa}\right/\frac{d\nu}{d\alpha}$. Since for each index $i$ in ``$\sum_<$" we have $|x_i-\kappa|\leq\frac{1}{\sqrt{\alpha}}$, we may bound $\frac{d\muhat}{d\alpha}$, involving a $3^\textrm{rd}$ moment term, by the simpler $|\frac{d\muhat}{d\alpha}|\leq \frac{1}{n}\sum_< |x_i-\kappa|^2/\sqrt{\alpha}$. Combining this, with the other derivatives and the bound $\alpha\leq \frac{1}{3}\left.\log\frac{1}{\delta}\right/\sum_< (x_i-\kappa)^2$ from the previous paragraph yields: \[\left|\frac{d\muhat}{d\alpha}\left.\frac{d\nu}{d\kappa}\right/\frac{d\nu}{d\alpha}\right|\leq\frac{2\sqrt{\alpha}}{n}\left|\frac{\sum_< (x_i-\kappa)\sum_< (x_i-\kappa)^2}{\sum_< (x_i-\kappa)^2}\right|\leq\frac{2\sqrt{\frac{1}{3}\log\frac{1}{\delta}}}{n}\left|\frac{\sum_< (x_i-\kappa)}{\sqrt{\sum_< (x_i-\kappa)^2}}\right|\leq \sqrt{\frac{4\log\frac{1}{\delta}}{3n}}\]
where the last inequality is Cauchy-Schwarz applied to the sequence $(x_i-\kappa)$ and the all-1s sequence.
\end{proof}

\end{document}